\documentclass{amsart} 

\usepackage[english]{babel}
\usepackage[latin1]{inputenc}

\usepackage{amsmath}
\usepackage{amsthm}
\usepackage{amssymb}
\usepackage{tikz}

\usepackage{imakeidx} 
\usepackage{appendix}
\usepackage{tikz-cd}
\usepackage{verbatim}
\usepackage{enumitem}
\usepackage{multicol}
\usepackage{textgreek}

\usepackage[backref=page]{hyperref}
\usepackage{cleveref}

\hypersetup{colorlinks=true,linkcolor=blue,anchorcolor=blue,citecolor=blue}

\usepackage{adjustbox}

\usepackage{mathtools}

\newtheorem{thm}{Theorem}[section]
\newtheorem{prop}[thm]{Proposition}
\newtheorem{lemma}[thm]{Lemma}
\newtheorem{cor}[thm]{Corollary}

\theoremstyle{definition}
\newtheorem{defin}[thm]{Definition}

\theoremstyle{remark}
\newtheorem{rmk}[thm]{Remark}
\numberwithin{equation}{section}

\newcommand{\Q}{\mathbb Q}
\newcommand{\E}{\mathbb E}
\newcommand{\F}{\mathbb F}
\newcommand{\C}{\mathbb C}
\newcommand{\Z}{\mathbb Z}

\renewcommand{\P}{\mathbb P}
\newcommand{\Spec}{\operatorname{Spec}}

\newcommand{\mc}[1]{\mathcal{#1}}
\newcommand{\cl}{\overline}
\newcommand{\set}[1]{\left\{#1\right\}}
\renewcommand{\phi}{\varphi}

\newcommand{\on}[1]{\operatorname{#1}}
\newcommand{\ang}[1]{\left \langle{#1}\right \rangle}

\title[Varieties over $\cl{\Q}$ with infinite Chow groups modulo almost all primes]{Varieties over $\cl{\Q}$ with infinite Chow groups modulo almost all primes}

\address{Department of Mathematics\\
	University of California Los Angeles\\
	CA 90095 \\
	United States of America}

\author{Federico Scavia}
\email{scavia@math.ucla.edu}

\makeatletter
\@namedef{subjclassname@2020}{\textup{2020} Mathematics Subject Classification}
\makeatother

\date{June 2023}


\begin{document}

	\begin{abstract}
		Let $E$ be the Fermat cubic curve over $\cl{\Q}$. In 2002, Schoen proved that the group $CH^2(E^3)/\ell$ is infinite for all primes $\ell\equiv 1\pmod 3$. We show that $CH^2(E^3)/\ell$ is infinite for all prime numbers $\ell> 5$. This gives the first example of a smooth projective variety $X$ over $\overline{\mathbb{Q}}$ such that $CH^2(X)/\ell$ is infinite for all but at most finitely many primes $\ell$. A key tool is a recent theorem of Farb--Kisin--Wolfson, whose proof uses the prismatic cohomology of Bhatt--Scholze.
	\end{abstract}
	
	\maketitle

	\section{Introduction}

	Let $k$ be a field and $X$ be a smooth projective variety over $k$. If $k$ is a number field, the motivic Bass conjecture predicts that the Chow groups $CH^i(X)$ are finitely generated for all $i\geq 0$. In particular, $CH^i(X)/n$ should be finite for all $i\geq 0$ and all integers $n\geq 1$. 
	
	Suppose now that $k$ is an algebraically closed field of characteristic zero. The groups $CH^i(X)/n$ are finite if $i\in\set{0,1,\dim(X)}$. In response to a question of Colliot-Th\'el\`ene, Schoen \cite{schoen2002complex} proved that $CH^i(X)/n$ can be infinite for any $i$ such that $2\leq i\leq \dim(X)-1$. Recall that the Fermat cubic curve is the elliptic curve $E \subset \P^2_{\Q}$ given by the homogeneous equation $x_0^3 + x_1^3+x_2^3=0$.
	
	\begin{thm}[Schoen]\label{schoen-thm}
		Let $E\subset \P^2_{\Q}$ be the Fermat cubic curve. Then $CH^2(E^3_F)/\ell$ is infinite for all primes $\ell\equiv 1 \pmod 3$ and all algebraically closed fields $F$ of characteristic zero.
	\end{thm}
	Note that the primes $\ell\equiv 1\pmod 3$ are exactly the primes of good ordinary reduction for $E^3$. By the Rigidity Theorem of Lecomte and Suslin \cite{lecomte1986rigidite}, in order to prove \Cref{schoen-thm} one may assume that $F=\cl{\Q}$.
	
	After \Cref{schoen-thm}, many examples of complex varieties with infinite Chow groups modulo primes were found. Rosenschon--Srinivas \cite{rosenschon2010griffiths} proved that for a very general principally polarized complex abelian threefold $A$, there exists a constant $\ell_0>0$ such that $CH^2(A)/\ell$ is infinite for all primes $\ell\geq \ell_0$, but gave no upper bound on $\ell_0$. Totaro \cite{totaro2016complex} later showed that $CH^2(A)/\ell$ is infinite for all prime numbers $\ell$, thus exhibiting the first example of a smooth projective complex variety $X$ such that $CH^2(X)/\ell$ is infinite for all primes $\ell$. This was recently generalized by Diaz \cite{diaz2021nondivisible}, who proved that $CH^i(E_1\times\dots \times E_n)/\ell$ is infinite for all $n\geq 3$, all $2\leq i\leq n-1$, and all very general complex elliptic curves $E_1,\dots,E_n$. There are also very recent examples of complex fivefolds $X$ such that the Griffiths group $\on{Griff}^3(X)$ has infinite $\ell$-torsion, and in fact even $\on{Griff}^3(X)[\ell]/\ell \on{Griff}^3(X)[\ell^2]$ is infinite: this is due to Schreieder \cite{schreieder2020infinite} for $\ell=2$, and to Alexandrou \cite{alexandrou2023torsion} for $\ell>2$. Alexandrou's paper relies on Schreieder's, for example on the injectivity theorem \cite[Theorem 6.1]{schreieder2020infinite} that works for surfaces admitting certain special degenerations, and hence works for arbitrary integers. Alexandrou's contribution was to construct such surfaces, generalizing the construction used by Schreieder when $\ell=2$.
	
	All the previous examples, with the exception of Schoen's, depend on the existence of parameters in their field of definition, and so leave the following question open: Is there a smooth projective variety $X$ over $\cl{\Q}$ such that $CH^2(X)/\ell$ is infinite for all sufficiently large $\ell$? In particular, the question of whether the group $CH^3(E^3_{\cl{\Q}})/\ell$ is infinite for the Fermat cubic $E/\Q$ and all sufficiently large primes numbers $\ell$ had remained open. We answer these questions affirmatively.
	
	\begin{thm}\label{thm-fermat}
		Let $E\subset \P^2_{\Q}$ be the Fermat cubic curve. Then $CH^2(E^3_F)/\ell$ is infinite for all primes $\ell> 5$ and all algebraically closed fields $F$ of characteristic zero.
	\end{thm}
	The conclusion of \Cref{thm-fermat} is false in positive characteristic. Indeed, let $p\neq 3$ and $\ell\neq p$ be prime numbers, and write $E_p$ for the reduction of $E$  modulo $p$. By \cite[Theorem (0.2)(ii)]{schoen2002complex}, for every algebraically closed field $L$ of characteristic $p$, the group $CH^2((E_p)^3_L)/\ell$ is finite.
	
	It follows from \Cref{thm-fermat} and \cite[Remarks (14.1) and (14.2)]{schoen2002complex} that the quotient modulo $\ell$ of the Griffiths group $\on{Griff}^2(E^3_F)/\ell$ and the unramified cohomology group $H^3_{\on{nr}}(F(E^3),\Z/\ell(2))$ are also infinite for all primes $\ell> 5$ and all algebraically closed fields $F$ of characteric zero.
	
	The combination of \Cref{thm-fermat} with the existing literature gives several remarkable corollaries.
	
	\begin{cor}\label{standard-cor}
		For each $n \geq 3$, there is a smooth projective $n$-fold $X$ over $\Q$ such that $CH^i(X_{\cl{\Q}})/\ell$ is infinite for all $2 \leq i \leq n - 1$ and all prime numbers $\ell>5$.
	\end{cor}
	Indeed, by \Cref{thm-fermat} and the projective bundle formula for Chow groups, one may take $X=E^3\times_{\Q}\P^{n-3}_{\Q}$. 
	
	\begin{cor}\label{standard-cor2}
		Let $\Q(t)$ be a purely transcendental extension of $\Q$, of transcendence degree $1$. For all $n \geq 4$, there is a smooth projective $n$-fold $X$ over $\Q(t)$ such that $CH^i(X_{\cl{\Q(t)}})[\ell]$ is infinite for all $3 \leq i \leq n - 1$ and all prime numbers $\ell> 5$.
	\end{cor}
	The condition $i\neq 2$ is necessary: if $V$ is a smooth projective variety over an algebraically closed field $F$ of characteristic zero, the group $CH^2(V)[\ell]$ is finite for all primes $\ell$. This follows from the Merkurjev--Suslin Theorem; see \cite[Corollary (18.1)]{merkurjev1982k}.

	\begin{cor}\label{cor3}
		Let $K=\Q(\zeta_\ell)(t)$, where $\zeta_{\ell}$ is a primitive $\ell$-th root of unity and $t$ is a variable. For all $n\geq 5$, there is a smooth projective $n$-fold $X$ over $K$ such that the $i$-th Griffiths group $\on{Griff}^i(X_{\cl{K}})[\ell]$ is infinite for all $3\leq i\leq n-1$ and all prime numbers $\ell>5$.
	\end{cor}

	\medskip
	
	A key ingredient for the proof of \Cref{schoen-thm} is a result of Bloch--Esnault \cite[Theorem 1.2]{bloch1996coniveau}, which implies that the restriction map \[H^i(X,\Z/p)\to H^i(\cl{\Q}(X),\Z/p)\] is non-zero, when $X$ is a smooth projective $\cl{\Q}$-variety with good \emph{ordinary} reduction at a prime $p$, the reduction of $X$ modulo $p$ has a non-zero global differential $i$-form, and either $p>i+2$ or the crystalline cohomology of $X$ is $p$-torsion free. Schoen applied this theorem to the desingularization of the self-fiber product of a certain elliptic surface $W$ related to $E^3$, whose construction we recall in \Cref{construction}.
	
	The proof of \Cref{thm-fermat} is made possible by a recent theorem of Farb--Kisin--Wolfson \cite[Theorem 4]{farb2021prismatic}, whose proof uses recent advances in prismatic cohomology, and which implies the same conclusion as the theorem of Bloch--Esnault, but for all primes $p > i+2$ of good unramified reduction, not necessarily ordinary; see \Cref{fkw}. 
	
	\subsection*{Notation}
	If $A$ is an abelian group and $n\geq 1$ is an integer, we denote by $A[n]$ and $A/n$ the kernel and cokernel of the homomorphism $A\to A$ given by multiplication by $n$, respectively. We also write $A/\on{tors}$ for the quotient of $A$ by its torsion subgroup.
	
	If $k$ is a field, we write $\cl{k}$ for a separable closure of $k$ and $G_k\coloneqq \on{Gal}(\cl{k}/k)$ for the absolute Galois group of $k$. If $M$ is a continuous (left) $G_k$-module and $i\geq 0$ is an integer, we write $H^i(k,M)=H^i(G_k,M)$ for the $i$-th Galois cohomology group of $M$.
	
	A $k$-variety $X$ is a separated geometrically integral $k$-scheme of finite type. We write $k(X)$ for the function field of $X$. 
	
	Let $X$ be a smooth variety over $k$. For all integers $i\geq 0$, we let $Z^i(X)$ be the free abelian group of codimension $i$ cycles on $X$, and we write $CH^i(X)$ for the Chow group of codimension $i$ cycles on $X$, that is, the quotient of $Z^i(X)$ modulo rational equivalence. 
	
	We will not need to consider \'etale cohomology of sheaves on $X$, but only on $X_{\cl{k}}$. If $\mc{F}$ is an \'etale sheaf on $X$, we write $H^i(X_{\cl{k}},\mc{F})$ for the cohomology of the pullback of $\mc{F}$ to the small \'etale site of $X_{\cl{k}}$. If $Z\subset X$ is a closed subscheme of $X$, we write $H^i_Z(X_{\cl{k}},\mc{F})$ for the cohomology of the pullback of $\mc{F}$ supported on $Z_{\cl{k}}$. We also set
	\[H^i_Z(X_{\cl{k}},\mc{F})_0\coloneqq \on{Ker}[H^i_Z(X_{\cl{k}},\mc{F})\to H^i(X_{\cl{k}},\mc{F})].\]
	If $Z$ consists of a single closed point $x$, we will write $H^i_x(X_{\cl{k}},\mc{F})$ and $H^i_x(X_{\cl{k}},\mc{F})_0$ for $H^i_Z(X_{\cl{k}},\mc{F})$ and $H^i_Z(X_{\cl{k}},\mc{F})_0$, respectively. If $\ell$ is a prime number invertible in $k$ and $j$ is an integer, these definitions apply in particular to the \'etale cohomology with twisted coefficients $\Z/\ell(j)$.
	
	Similarly, we write  $H^i(X_{\cl{k}},\Z_{\ell}(j))$ (resp. $H^i_Z(X_{\cl{k}},\Z_{\ell}(j))$) for the $\ell$-adic cohomology of $X_{\cl{k}}$ (resp. the $\ell$-adic cohomology of $X_{\cl{k}}$ supported on $Z\subset X$) and set 
	\[H^i_Z(X_{\cl{k}},\Z_{\ell}(j))_0\coloneqq \on{Ker}[H^i_Z(X_{\cl{k}},\Z_\ell(j))\to H^i(X_{\cl{k}},\Z_\ell(j))].\]
	We let \[N^1H^i(X_{\cl{k}},\Z/\ell)\coloneqq \on{Ker}[H^i(X_{\cl{k}},\Z/\ell)\to H^i(\cl{k}(X),\Z/\ell)]\]
	be the first step of the coniveau filtration on $H^i(X_{\cl{k}},\Z/\ell)$. We write $H^0(X,\Omega^i)$ for the $k$-vector space of global differential $i$-forms on $X$.
	
	Suppose that $X$ is also projective. We write $\on{NS}(X)$ for the N\'eron-Severi group of $X$, that is, the quotient of $Z^1(X)$ by the subgroup of algebraically trivial cycles, in the sense of Fulton \cite[Definition 10.3]{fulton1998intersection}. (Note that we do not assume that $k$ is algebraically closed in this definition.) For all $i\geq 0$, we let $Z^i_{\on{hom},\ell}(X)$ be the kernel of the cycle map $\on{cl}^i\colon Z^i(X)\to H^{2i}(X_{\cl{k}},\Z_{\ell}(i))$, and let $CH^i_{\on{hom},\ell}(X)$ be the Chow group of homologically trivial cycles, that is, the quotient of $Z^i_{\on{hom},\ell}(X)$ modulo rational equivalence.

	\section{The theorem of Farb--Kisin--Wolfson}\label{fkw-section}
	Let $C$ be an algebraically closed field, $X$ be a smooth proper $C$-scheme, and $p$ be a prime number. Following \cite[2.2.12]{farb2021prismatic}, we say that $X$ has \emph{good reduction} at $p$ if there exist a $p$-adic valuation on $C$, with ring of integers $\mc{O}_C$, and a smooth proper $\mc{O}_C$-scheme $\mc{X}$ whose generic fiber is isomorphic to $X$. (Note that $C$ is not necessarily complete with the respect to the valuation.) We say that $X$ has \emph{unramified good reduction} at $p$ if $X$ has good reduction at $p$ and $\mc{X}$ can be chosen so that it is defined over a discrete valuation ring $R\subset \mc{O}_C$ whose maximal ideal is generated by $p$ (that is, $p$ is unramified in $R$) and whose residue field is perfect. 
	
	The key ingredient for our proof of \Cref{thm-fermat} is the following theorem of Farb--Kisin--Wolfson \cite[Theorem 4]{farb2021prismatic}. 
	
	\begin{thm}[Farb--Kisin--Wolfson]\label{fkw}
		Let $C$ be an algebraically closed field of characteristic zero, $i\geq 0$ be an integer, $p>i+2$ be a prime, and $X$ be a proper smooth connected scheme over $C$ with unramified good reduction at $p$. Then \[\dim_{\F_p}\on{Im}[H^i(X,\Z/p)\to H^i(C(X),\Z/p)]\geq \dim_CH^0(X,\Omega^i).\]
		In particular, if $H^0(X,\Omega^i)\neq 0$ then $N^1H^i(X,\Z/p)\neq H^i(X,\Z/p)$.
	\end{thm} 
	We will apply \Cref{fkw} to $C=\cl{\Q}$, $i=3$, $p>5$ and $X=W_{\cl{\Q}}$, where $W$ is the smooth projective threefold over $\Q$ defined in \Cref{construction}.

	\section{The threefold \texorpdfstring{$W$}{W}}\label{construction}
	The proof of \Cref{thm-fermat} is based on the study of the self-fiber product $W$ of a certain elliptic surface. We recall the construction of $W$ given by Schoen \cite[\S 10]{schoen2002complex}. 
	
	Let $k$ be a field of characteristic different from $2$ and $3$ and containing a primitive third root of unity $\zeta$. Let $u_0,u_1$ be homogeneous coordinates in $\P^1_k$ and $x_0,x_1,x_2$ be homogeneous coordinates on $\P^2_k$. We define a closed subscheme $Y \subset \mathbb{P}^1_k \times_k \mathbb{P}^2_k$ by the bihomogeneous equation
	\[u_0(x_0^3 + x_1^3 + x_2^3) - u_1x_0x_1x_2 = 0.\] 
	Let $\pi\colon Y\to \P^1_k$ be the projection onto the first factor. The morphism $\pi$ admits a section $s\colon \P^1_k\to Y$ whose image is $\mathbb{P}^1_k \times_k \set{(0 : -1 : 1)}$. A general fiber of $\pi$ is a smooth cubic curve in $\P^2_k$, and hence has genus one. There are four singular fibers of $\pi$, namely those of 
	\[(1:3),\quad (1:3\zeta),\quad (1:3\zeta^2),\quad (1:0).\] Each singular fiber has Kodaira type $I_3$, that is, it consists of a triangle of smooth rational curves with nodal intersections. Thus $\pi$ is a semistable elliptic surface. In fact, $\pi$ is the relatively minimal, regular, projective model of the universal elliptic curve with a symplectic level $3$ structure; see \cite[\S 1]{schoen1993complex}.

	Let 
	\[U\coloneqq \P^1_k-\set{(1:3),(1:3\zeta),(1:3\zeta^2),(1:0)}\] be the smooth locus of $\pi$, let $j:U\hookrightarrow \P^1_k$ be the inclusion morphism, and let $\pi_U\colon Y|_U\to U$ and be the restriction of $\pi$. Thus $\pi_U$ is an abelian scheme with identity section given by the restriction of $s$ to $U$.
	
	Let $W \rightarrow Y \times_{\P^1_k} Y$ be the blow-up of the self-fiber product of $\pi$ along the singular locus of $Y\times_{\P^1_k}Y$. Since the singularities of $Y\times_{\P^1_k}Y$ are isolated ordinary double points, $W$ is smooth over $k$. We will view $W$ as a $\P^1_k$-scheme via the composition \[q\colon W\to Y\times_{\P^1_k}Y\to \P^1_k.\]

	\section{The \texorpdfstring{$\ell$}{l}-adic Abel-Jacobi map and CM cycles}
	
	We maintain the notation introduced in \Cref{construction}. 
	
	\subsection{The \texorpdfstring{$\ell$}{l}-adic Abel-Jacobi map}\label{abel-jacobi-def}
	Let $\ell$ be a prime number invertible in $k$. The $\ell$-adic intermediate Jacobian for cycles on $W$ of codimension $2$ is defined as the direct limit
	\[J_{\ell}^2(W)\coloneqq \varinjlim_{k'}H^1(k',H^3(W_{\cl{k}},\Z_{\ell}(2))/\on{tors}),\]
	where $k'$ ranges over all finite extensions of $k$ contained in $\cl{k}$. We refer the reader to \cite[\S 3]{schoen2002complex} for the definition of the $\ell$-adic Abel-Jacobi map
	\[\alpha_{\ell,k'}^2\colon CH^2_{\on{hom},\ell}(W_{k'})\to H^1(k',H^3(W_{\cl{k}},\Z_{\ell}(2))/\on{tors}),\]
	where $k\subset k'\subset \cl{k}$ is an intermediate field of finite degree over $k$. The definition of $\alpha_{\ell,k'}^2$ is functorial in $k'$. Passing to the direct limit in $k'$, we obtain a homomorphism
	\[\alpha_{\ell}^2\colon CH^2_{\on{hom},\ell}(W_{\cl{k}})\to J_{\ell}^2(W).\]
	For every intermediate field $k\subset k'\subset \cl{k}$ of finite degree over $k$, the abelian group $H^1(k',H^3(W_{\cl{k}},\Z_{\ell}(r))/\on{tors})$ is naturally a continuous $G_k$-module. Therefore, $J_{\ell}^2(W)$ is also a continuous $G_k$-module. 
	
	By construction, for every intermediate field $k\subset k'\subset \cl{k}$ of finite degree over $k$, we have a commutative square of $G_k$-modules
	\begin{equation}\label{abel-jacobi-diagram}
		\begin{tikzcd}
			CH^2_{\on{hom},\ell}(W_{k'}) \arrow[r,"\alpha_{\ell,k'}^2"] \arrow[d] & H^1(k',H^3(W_{\cl{k}},\Z_{\ell}(2))/\on{tors}) \arrow[d,hook] \\
			CH^2_{\on{hom},\ell}(W_{\cl{k}}) \arrow[r,"\alpha_{\ell}^2"] & J_{\ell}^2(W),
		\end{tikzcd}
	\end{equation}
	where by \cite[Lemma 3.1]{schoen2002complex} (which adapts an argument of \cite[4.1.1]{bloch1996coniveau}) the left vertical map is injective and has image equal to $J_{\ell}^2(W)^{G_{k'}}$.

	\subsection{CM points and CM cycles}\label{cm-point-cycles}
	
	Let $k'/k$ be a field extension and $x$ be a $k'$-rational point of $U_{k'}$. Following \cite[Definition (2.1)]{schoen2002complex}, we say that  $x$ is a \emph{CM point} if there exists an order $\mc{O}$ in an imaginary quadratic field  such that 
	\[\on{End}(\pi^{-1}(x))\cong \on{End}(\pi^{-1}(x)_{\cl{k'}})\cong \mc{O}.\]
	Note that $q^{-1}(x)=\pi^{-1}(x)\times_{k'} \pi^{-1}(x)$ for all field extensions $k'/k$ and all $x\in U_{k'}$. If $x\in U_{k'}$ is a CM point, then the N\'eron-Severi group $\on{NS}(q^{-1}(x))$ of $q^{-1}(x)$ has rank $4$; see \cite[Chapter IV, Exercise 4.10]{hartshorne} or \cite[(2.3)]{schoen2002complex}. Let 
	\[\on{NS}_0(q^{-1}(x))\coloneqq \on{Span}_{\Z}\set{\pi^{-1}(x)\times_{k'} s(x),s(x)\times_{k'} \pi^{-1}(x),\Sigma_{\pi^{-1}(x)}}\subset \on{NS}(q^{-1}(x)),\]
	where $\Sigma_{\pi^{-1}(x)}$ is the class of the diagonal of $q^{-1}(x)$. 
	
	Following \cite[Definition (2.1)]{schoen2002complex}, a cycle $z\in Z^2(W_{k'})$ is called a \emph{CM cycle} if there exists a CM point $x\in U_{k'}$ such that the support of $z$ is contained in $q^{-1}(x)$ and the class of $z$ in $\on{NS}(q^{-1}(x))$ generates the orthogonal complement of $\on{NS}_0(q^{-1}(x))$.
	
	\begin{lemma}\label{cm-homologically-trivial}
		Let $k\subset k'\subset \cl{k}$ be an intermediate field of finite degree over $k$ and let $z\in Z^2(W_{k'})$ be a CM cycle. Then $z$ lies in $Z^2_{\on{hom},\ell}(W_{k'})$ for all primes $\ell\geq 5$ invertible in $k$.
	\end{lemma}
	
	\begin{proof}
		See \cite[Lemma (2.6)]{schoen2002complex}.
	\end{proof}

	\subsection{The idempotent \texorpdfstring{$P$}{P}}\label{idempotent}
	In \Cref{aj-cm-cycles}, we will need to use a commutative diagram proved in \cite[Proposition (5.3)]{schoen2002complex}. In order to state this precisely, we introduce the following idempotent $P$.
	
	The inversion map in the abelian scheme $\pi_U$ extends to give an involution of $Y$ over $\P^1_k$, which we denote by $-1 \in \operatorname{Aut}(Y/\P^1_k)$. The subgroup of $\operatorname{Aut}((Y \times_{\P^1_k} Y)/\P^1_k)$ generated by $(-1, \operatorname{Id}_Y)$, $(\operatorname{Id}_Y, -1)$, and the permutation of the two factors in the fiber product $\tau$ is isomorphic to the dihedral group with $8$ elements. This group stabilizes the singular locus of $Y\times_{\P^1_k}Y$ and hence lifts to an isomorphic group $D \subset \operatorname{Aut}(W/\P^1_k)$.
	
	Let $\chi\colon D \to \set{\pm1}$ be the unique group homomorphism such that $\chi(\tau)=\chi((-1, \mathrm{Id}_Y))=-1$. We write $\Z[1/2][D]$ for the group algebra of $D$ over $\Z[1/2]$, and we define
	\[P \coloneqq \frac{1}{8}\sum_{g\in D} \chi(g)g \in \mathbb{Z}[1/2][D].\] 
	Note that $P$ is idempotent and $gP = Pg = \chi(g)P$ for all $g \in D$.
	
	\subsection{The Abel-Jacobi map on CM cycles}\label{aj-cm-cycles}
	
	Let $\ell\geq 5$ be a prime number invertible in $k$, let $k'/k$ be a field extension and $x$ be a $k'$-rational point of $U_{k'}$. 
	
	The idempotent $P$ introduced in \Cref{idempotent} defines a $G_{k'}$-equivariant projection $P_*\colon H^3(W_{\cl{k}},\Z_{\ell}(2))\to P_*H^3(W_{\cl{k}},\Z_{\ell}(2))$, hence a homomorphism
	\[P_*\colon H^1(k',H^3(W_{\cl{k}},\Z_{\ell}(2)))\to H^1(k',P_*H^3(W_{\cl{k}},\Z_{\ell}(2))).\]
	Consider the following \'etale sheaf on $\P^1_k$:
	\[\mc{M}\coloneqq j_*\on{Sym}^2(R^1(\pi_U)_*\Z/\ell(2)).\]
	Let $k\subset k'\subset \cl{k}$ be an intermediate field of finite degree over $k$ and let $x\in U_{k'}$ be a $k'$-rational point. Since $\ell\geq 5$, by \cite[Proposition (5.3)]{schoen2002complex} we have the following commutative diagram of continuous $G_{k'}$-modules with exact rows and surjective vertical maps:
	\begin{equation}\label{5.3-diag}
		\adjustbox{max width=\textwidth}{
			\begin{tikzcd}
				0 \arrow[r] & P_*H^3(W_{\overline{k}},\Z_{\ell}(2)) \arrow[r] \arrow[d,->>,"\eta"] & P_*H^3((W_{k'}-q^{-1}(x))_{\overline{k}},\Z_{\ell}(2)) \arrow[r] \arrow[d,->>] & P_*H^4_{q^{-1}(x)}(W_{\overline{k}},\Z_{\ell}(2))_0 \arrow[r]\arrow[d,->>] & 0 \\
				0 \arrow[r] & H^1(\P^1_{\overline{k}}, \mc{M}) \arrow[r] & H^1((\P^1_{k'}-\set{x})_{\overline{k}}, \mc{M}) \arrow[r] & H^2_x(\P^1_{\overline{k}}, \mc{M}) \arrow[r] & 0.
			\end{tikzcd}
		}
	\end{equation}
	The top (resp. bottom) row of (\ref{5.3-diag}) comes from the Gysin sequence \cite[Remark 5.4(b)]{milne1980etale} for the base change to $\cl{k}$ of the inclusion $\set{x}\hookrightarrow\P^1_{k'}$ (resp. $q^{-1}(x)\hookrightarrow W_{k'}$). Note that, since $\ell\geq 5$, we have $H^2(\P^1_{\cl{k}},\mc{M})=0$ by \cite[Corollary (4.5)(vi)]{schoen2002complex}, and hence $H^2_x(\P^1_{\overline{k}}, \mc{M})=H^2_x(\P^1_{\overline{k}}, \mc{M})_0$.
	
	The homomorphism $\eta$ induces a map
	\[
	\eta_*\colon H^1(k',P_*H^3(W_{\cl{k}},\Z_{\ell}(2)))\to H^1(k',H^1(\P^1_{\cl{k}},\mc{M})).
	\]
	We define 
	\begin{equation}\label{Delta-define}
		\Delta_{k'}\coloneqq \eta_*\circ P_*\colon H^1(k',H^3(W_{\cl{k}},\Z_{\ell}(2)))\to H^1(k',H^1(\P^1_{\cl{k}},\mc{M})).
	\end{equation}
	The map $\Delta_{k'}$ appears without name in the statement of \cite[Lemma (5.4)(ii)]{schoen2002complex}. Since $P_*$ and $\eta_*$ are compatible with restrictions along finite field extensions $k''/k'$, so is $\Delta_{k'}$. We also let
	\[
	\delta_x\colon H_x^2(\P^1_{\cl{k}},\mc{M})^{G_k}\to H^1(k', H^1(\P^1_{\cl{k}},\mc{M})).
	\]
	be the connecting homomorphism induced by the bottom row of (\ref{5.3-diag}).
	
	Assume now that $x \in U_{k'}$ is a CM point and that $z$ is a CM cycle supported on $q^{-1}(x)$. We write $[z]$ for the cohomology class of $z$ in $H^2(q^{-1}(x)_{\cl{k}}, \mathbb{Z}_l(1))^{G_{k'}}$. By \cite[Proposition (2.5)(v) and Lemma (4.3)(iii)]{schoen2002complex}, we have $P_*[z]=[z]$ in $H^2(q^{-1}(x)_{\cl{k}}, \mathbb{Z}_l(1))^{G_{k'}}$. Using the canonical identification
	\[
	H^2(q^{-1}(x)_{\cl{k}}, \mathbb{Z}_{\ell}(1)) \xrightarrow{\sim} H^4_{q^{-1}(x)}(W_{\cl{k}}, \mathbb{Z}_{\ell}(2)),
	\]
	we view $[z]$ as a ${G_{k'}}$-invariant element of $P_*H^4_{q^{-1}(x)}(W_{\cl{k}}), \mathbb{Z}_{\ell}(2))$. By \Cref{cm-homologically-trivial}, the element $[z]$ even lies in $P_*H^4_{q^{-1}(x)}(W_{\cl{k}}), \mathbb{Z}_{\ell}(2))_0$. We may thus consider the element  $\theta_x([z]) \in H^2_x(\P^1_{\cl{k}}, \mc{M})^{G_{k'}}$, where $\theta_x$ has been defined in (\ref{5.3-diag}). 
	
	\begin{lemma}\label{5.4}
		Let $k\subset k'\subset \cl{k}$ be a finite extension of $k$, let $x \in U_{k'}$ be a CM point, and let $z$ be a CM cycle in $q^{-1}(x)$. Let $\ell\geq 5$ be a prime number invertible in $k$.
		
		(i) We have $\theta_x([z])\neq 0$ in $H^2_x(\P^1_{\cl{k}}, \mc{M})^{G_{k'}}$ is not zero.
		
		(ii) We have $\Delta_{k'}(\alpha^2_{\ell,k'}(z))=\delta_x(\theta_x([z]))$.
	\end{lemma}
	
	\begin{proof}
		See \cite[Lemma (5.4)]{schoen2002complex}. Note that in \cite{schoen2002complex} the maps $\Delta_{k'}$ and $\theta_x$ are not named, and $\theta_x([z])$ is denoted $[\cl{z}]$.
	\end{proof}
	
	\begin{rmk}
		The idempotent $P$ acts as the identity on $H^3(W_{\cl{K}},\Z_{\ell}(2))$; see \cite[Lemma 10.3(iii)]{schoen2002complex}. Therefore, the homomorphism $P_*$ used in (\ref{Delta-define}) is equal to the identity, and so $\Delta_{k'}=\eta_*$. Our reason for introducing $P$ in this paper is to get (\ref{5.3-diag}) by invoking the general result \cite[Proposition (5.3)]{schoen2002complex}.
	\end{rmk}
	
	\section{Good reduction and coniveau}\label{good-reduction}
	
	We maintain the notation of \Cref{construction}. 
	
	\begin{prop}\label{good-reduction-prop}
		Suppose that $k=K$ is a number field, let $\mathfrak{p}$ be a finite place of $K$ of residue characteristic $p\geq 5$, and write $R$ for the valuation ring of $\mathfrak{p}$ and $\F$ for the residue field of $\mathfrak{p}$. Then the elliptic surface $\pi\colon Y\to \P^1_K$ has good reduction at $\mathfrak{p}$, in the sense of \cite[Definition 7.2]{schoen2002complex}. Moreover, the $K$-variety $W$ satisfies the following properties.
		\begin{enumerate}
			\item The $J$-invariant $J\colon \P^1_K\rightarrow \mathbb{P}^1_K$ extends to a finite morphism $\mc{J}\colon \P^1_R\rightarrow \mathbb{P}^1_R$.
			\item The restriction $\mc{J}_{\F}:\mathbb{P}^1_{\F}\rightarrow \mathbb{P}^1_{\F}$ of $\mc{J}$ is finite and separable.
			\item There exists a flat morphism $\Pi:\mc{Y}\rightarrow \mathbb{P}^1_R$ of smooth, proper, irreducible $R$-schemes such that $\Pi_K=\pi$.
			\item The section $s$ of $\pi$ extends to a section $\mc{S}$ of $\Pi$.
			\item The restriction $\Pi_{\F}:\mc{Y}_{\F}\rightarrow \P^1_{\F}$ of $\Pi$ is a relatively minimal, non-isotrivial, semi-stable elliptic surface with the section $\mc{S}_{\F}$.
			\item There are exactly four singular fibers of $\Pi_{\F}$, namely those of $(1:3)$, $(1:3\zeta)$, $(1:3\zeta^2)$, $(1:0)$, and each singular fiber consists of a triangle of smooth rational curves.
			\item The blow-up $\mc{W}$ of the reduced singular locus of $\mc{Y}\times_{\P^1_R} \mc{Y}$ is smooth over $R$.
			\item The $D$-action on $W$ described in \Cref{idempotent} uniquely extends to a $D$-action on $\mc{W}$.
		\end{enumerate}
		In particular, $W_{\cl{\Q}}$ has good unramified reduction at $p$, in the sense of \Cref{fkw-section}.
	\end{prop}
	
	\begin{proof}
		The fact that $\pi$ has good reduction in the sense of \cite[Definition (7.2)]{schoen2002complex} follows from \cite[Table 5.3 X3333]{miranda1986extremal}. Properties (1)-(8) are proved in  \cite[Proposition (7.4)]{schoen2002complex}, except for the finiteness of $\mc{J}$ in (1), which follows from the rest. Indeed, $\mc{J}$ is quasi-finite because $\mc{J}_K$ and $\mc{J}_{\F}$ are finite, and it is proper because its domain $\P^1_R$ is projective over $R$. 
		
		Property (7) implies that $W_{\cl{\Q}}$ has good reduction at $p$ in the sense of \Cref{fkw-section}. The variety $W_{\cl{\Q}}$ is defined over $\Q(\zeta)$, where $\zeta$ is a primitive third root of unity. Since $p$ is at least $5$, it does not ramify in $\Q(\zeta)$, hence $W_{\cl{\Q}}$ has good unramified reduction at $p$.
	\end{proof}

	\begin{lemma}\label{n1=0}
		Let $k=K=\Q(\zeta)$, where $\zeta$ is a primitive third root of unity.
		Then $N^1H^3(W_{\cl{K}},\Z/\ell(2))=0$ for all primes $\ell\geq 5$ such that $\ell\equiv 2 \pmod 3$.
	\end{lemma}
	
	\begin{proof}
		Let $\ell\geq 5$ be a prime number. By \Cref{good-reduction-prop}, $W$ has good unramified reduction at $\ell$.  By \cite[Lemma 6.5]{schoen2002complex}, we have $H^0(W,\Omega^3)\neq 0$, hence $H^0(W_{\cl{K}},\Omega^3)\neq 0$. By \Cref{fkw} applied to $p=\ell$, we deduce that
		\begin{equation*}
			N^1H^3(W_{\cl{K}},\Z/\ell(2))\neq H^3(W_{\cl{K}},\Z/\ell(2)).\end{equation*}
		By \cite[Lemma (10.3)(i)]{schoen2002complex}, the $(\Z/\ell)$-vector space $H^3(W_{\cl{K}},\Z/\ell(2))$ has dimension $2$, therefore the subspace $N^1H^3(W_{\cl{K}},\Z/\ell(2))$ is either trivial or $1$-dimensional.
		
		By \cite[Lemma (10.3)(iv)]{schoen2002complex}, there exists an element $\sigma\in \on{Aut}(W_K)$ of order $3$ and such that the $(\Z/\ell)$-linear automorphism \[\sigma^*\colon H^3(W_{\cl{K}},\Z/\ell(2))\to H^3(W_{\cl{K}},\Z/\ell(2))\] induced by $\sigma$ does not have the eigenvalue $1$.  Since $\ell\equiv 2\pmod 3$, there are no elements of order $3$ in $(\Z/\ell)^\times$, hence $\sigma^*$ has no eigenvectors over $\Z/\ell$.
		
		Suppose that $N^1H^3(W_{\cl{K}},\Z/\ell(2))$ is $1$-dimensional, and let $v$ be a generator of $N^1H^3(W_{\cl{K}},\Z/\ell(2))$. By the functoriality of the coniveau filtration, $\sigma^*(v)$ also belongs to $N^1H^3(W_{\cl{K}},\Z/\ell(2))$, and so $\sigma^*(v)=cv$ for some $c\in (\Z/\ell)^{\times}$. This is a contradiction because $\sigma^*$ admits no eigenvectors over $\Z/\ell$. We conclude that $N^1H^3(W_{\cl{K}},\Z/\ell(2))=0$, as desired. 
	\end{proof}
	
	\begin{rmk}
		The proof of \Cref{n1=0} fails for all primes $\ell\equiv 1\pmod 3$. Indeed, by \cite[Lemma (10.3)(vi)]{schoen2002complex} the $G_K$-representation $H^3(W_{\cl{K}},\Z/\ell(2))$ is reducible for all primes $\ell\equiv 1\pmod 3$.
	\end{rmk}
	
	We have introduced the Abel-Jacobi map $\alpha_{\ell}^2$ in \Cref{abel-jacobi-def}.
	
	\begin{prop}\label{6.2}
		Suppose that $k=K=\Q(\zeta)$, where $\zeta$ is a primitive third root of unity. 
		\begin{itemize}
			\item[--] Let $K\subset K_1\subset \cl{\Q}$ be a finite extension of $K$.
			\item[--] Let $\ell\geq 5$ be a prime number such that $\ell\equiv 2\pmod 3$.
			\item[--] Let $z\in Z^2_{\on{hom},\ell}(W_{K_1})$ be such that the image of $\alpha^2_{\ell,K_1}(z)$ in $J_{\ell}^2(W)^{G_{K_1}}/\ell$ is non-zero.
		\end{itemize}
		Then the class of $z$ in $CH^2(W_{\cl{\Q}})/\ell$ is non-zero.
	\end{prop}
	
	\begin{proof}
		By \cite[Proposition (4.15)]{schoen2002complex}, the $\Z_{\ell}$-modules $H^3(W_{\cl{\Q}},\Z_{\ell})$ and $H^4(W_{\cl{\Q}},\Z_{\ell})$ are torsion-free. By \Cref{n1=0}, the group $N^1H^3(W_{\cl{\Q}},\Z/\ell(2))$ is trivial. The conclusion follows from \cite[Proposition (6.2)]{schoen2002complex}, applied to $V=W$ and $Q=\on{Id}_W$.
	\end{proof}
	
	\begin{rmk}
		Under the assumptions of \Cref{6.2}, it is clear that the class of $z$ in $CH^2_{\on{hom},\ell}(W_{K_1})/\ell$ is non-zero. The difficulty is in showing that the class of $z$ remains non-zero in $CH^2(W_{\cl{\Q}})/\ell$ (or equivalently in $CH^2_{\on{hom},\ell}(W_{\cl{\Q}})/\ell$, since $H^4(W_{\cl{\Q}},\Z_{\ell})$ is torsion-free. The main point of the argument of \cite[Proposition (6.2)]{schoen2002complex} (which goes back to \cite[Proposition (4.1)]{bloch1996coniveau}) is to prove that $CH^2(W_{\cl{\Q}})[\ell]$ is zero. This is a consequence of the Merkurjev--Suslin Theorem; see \cite[Proposition (6.1)]{schoen2002complex}.
	\end{rmk}

	We conclude this section by discussing the compatibility of $\delta_x$  with specialization. We let $\mc{U}\subset \P^1_R$ be the smooth locus of $\Pi$ and write $\underline{j}\colon\mc{U}\to\P^1_R$ for the inclusion morphism. Note that $\underline{j}$ restricts to $j$ on the generic fiber $U$. We define
	\[\tilde{\mc{M}}\coloneqq \underline{j}\vphantom{j}_*\on{Sym}^2(R^1\Pi_*\Z/\ell(2)).\]
	Then $\tilde{\mc{M}}$ is a sheaf on $\P^1_R$ which restricts to $\mc{M}$ on $\P^1_K$. We define $\mc{M}_0$ as the restriction of $\tilde{\mc{M}}$ to $\P^1_{\F}$.

	\begin{prop}\label{7.9}
		Suppose that $k=K$ is a number field.
		\begin{itemize}
			\item[--] Let $\mathfrak{p}$ be a finite place of $K$ of residue characteristic $p\geq 5$, let $R$ be the valuation ring of $\mathfrak{p}$ and $\F$ be the residue field of $\mathfrak{p}$.
			\item[--] Fix a place $\cl{\mathfrak{p}}$ of $\cl{K}$ lying above $\mathfrak{p}$, and let $D\subset G_{K'}$ be the corresponding decomposition group of $\mathfrak{p}$.
			\item[--]  Let $t$ be an $R$-point of $\mc{U}_R$, and write $x\in U(K)$ and $u\in \mc{U}(\F)$ for the restrictions of $t$ to the generic and special fiber, respectively.
		\end{itemize} 
		Then we have a commutative diagram
		\begin{equation*}
			\begin{tikzcd}
				H^2_u(\P^1_{\cl{\F}},\mc{M}_0)^{G_{\F}} \arrow[r,"\delta_u"] \arrow[d,"\wr"] & H^1(\F, H^1(\P^1_{\cl{\F}},\mc{M}_0))\arrow[d,hook]  \\
				H^2_x(\P^1_{\cl{K}},\mc{M})^{D} \arrow[r,"\delta_x"] & H^1(D, H^1(\P^1_{\cl{K}},\mc{M})) \\
				H^2_x(\P^1_{\cl{K}},\mc{M})^{G_{K}} \arrow[r,"\delta_x"] \arrow[u,hook]  & H^1(K, H^1(\P^1_{\cl{K}},\mc{M})),\arrow[u] 
			\end{tikzcd}
		\end{equation*}
		where the horizontal maps have been defined in \Cref{aj-cm-cycles}, the top vertical maps are specialization maps, the bottom-left vertical map is the inclusion, and the bottom-right vertical map is the restriction map in Galois cohomology.
	\end{prop}
	
	\begin{proof}
		This is \cite[(7.9)]{schoen2002complex}.
	\end{proof}

	\section{Schoen sequences}
	We maintain the notation introduced in \Cref{construction}. We suppose further that $k=K$ is a number field, we let $\mathfrak{p}$ be a finite place of $K$ of residue characteristic $p\geq 5$ such that $p\neq \ell$, and we write $R$ and $\F$ for the valuation ring and residue field of $\mathfrak{p}$, respectively. We let
	\[R=R_0\subset R_1\subset R_2\subset\cdots,\]
	be an infinite tower of discrete valuation rings contained in $\cl{\Q}$, and for all $i\geq 0$ the morphism $\Spec(R_{i+1})\to \Spec(R_i)$ is \'etale. For all $i\geq 0$, we write $K_i$ for the fraction field of $R_i$ and $\F_i$ for the residue field of $R_i$, so that we have the towers
	\[K=K_0\subset K_1\subset K_2\subset\cdots,\qquad \F=\F_0\subset \F_1\subset \F_2\subset \cdots.\]
	For all $i\geq 1$, we write $\mathfrak{p}_i$ for the maximal ideal of $R_i$. We let $\cl{\mathfrak{p}}$ be a place of $\cl{K}$ lying above the $\mathfrak{p}_i$, and let $D_i\subset G_{K_i}$ be the decomposition group of $\mathfrak{p}$ corresponding to this choice. For all $i\geq 1$, we let $\Phi_i$ be the composition
	\begin{align*}
		\Phi_i \colon & CH^2_{\mathrm{hom},\ell}(W_{K_i}) \xrightarrow{\alpha^2_{\ell,K_i}} H^1(K_i,H^3(W_{\cl{K}},\mathbb{Z}_\ell(2))) \\
		& \xrightarrow{\Delta_{K_i}} H^1(K_i,H^1(\mathbb{P}^1_{\cl{K}},\mathcal{M})) \xrightarrow{r_i} H^1(D_i,H^1(\mathbb{P}^1_{\cl{K}},\mathcal{M})),
	\end{align*}
	where $\alpha_{\ell,K_i}^2$ is the Abel-Jacobi map of \Cref{abel-jacobi-def} for the extension $K_i/K$, the map $\Delta_{K_i}$ coincides with the map (\ref{Delta-define}) for the extension $K_i/K$, and $r_i$ is the restriction homomorphism in Galois cohomology. We have used the fact that, since $\ell\geq 5$, by \cite[Proposition (4.15)]{schoen2002complex} the $\Z_\ell$-module $H^3(W_{\cl{K}},\mathbb{Z}_\ell(2))$ is torsion free. 
	
	For all $1\leq j< i$, since $\mathfrak{p}_j\subset \mathfrak{p}_i\subset \cl{\mathfrak{p}}$, we have $D_i\subset D_j$, and hence a commutative square
	\[
	\begin{tikzcd}
		H^1(K_j,H^1(\P^1_{\cl{K}},\mc{M})) \arrow[r,"r_j"] \arrow[d] & H^1(D_j,H^1(\P^1_{\cl{K}},\mc{M})) \arrow[d] \\
		H^1(K_i,H^1(\P^1_{\cl{K}},\mc{M})) \arrow[r,"r_i"]  & H^1(D_i,H^1(\P^1_{\cl{K}},\mc{M})),
	\end{tikzcd}
	\]
	where vertical arrows are restriction maps in Galois cohomology. As we have mentioned in \Cref{abel-jacobi-def} and \Cref{aj-cm-cycles}, the maps $\alpha^2_{\ell,k'}$ and $\Delta_{k'}$ are compatible with finite field extensions $k''/k'$. We thus obtain, for all $1\leq j<i$, a commutative square
	\begin{equation}\label{phi-comp}
		\begin{tikzcd}[column sep=large]
			CH^2_{\on{hom},\ell}(W_{K_j}) \arrow[r,"\Phi_j"] \arrow[d, hook] & H^1(D_j,H^1(\P^1_{\cl{K}},\mc{M})) \arrow[d] \\
			CH^2_{\on{hom},\ell}(W_{K_i})  \arrow[r,"\Phi_i"] & H^1(D_i,H^1(\P^1_{\cl{K}},\mc{M})),
		\end{tikzcd}
	\end{equation}
	where the left vertical map is given by pullback, and the vertical map is the restriction in Galois cohomology.
	
	Since $p\geq 5$, \Cref{good-reduction-prop} applies. We fix morphisms $\Pi\colon \mc{Y}\to \P^1_R$ and $\mc{J}\colon\P^1_R\to \P^1_R$ extending $\pi$ and $J$, respectively, as in \Cref{good-reduction-prop}. Recall that we denote by $\mc{U}\subset \P^1_R$ the smooth locus of $\Pi$. 
	
	For all $i\geq 1$, let $t_i$ be an $R_i$-point of $\mc{U}_{R_i}$, let $x_i\in U_{K_i}(K_i)$ be the restriction of $t_i$ to the generic fiber, and let $u_i\in \mc{U}_{\F_i}(\F_i)$ be the restriction of $t_i$ to the special fiber. For all $i\geq 1$, suppose that $x_i$ and $u_i$ are CM points, in the sense of \Cref{cm-point-cycles}.
	
	\begin{defin}\label{schoen-sequence}
		We say that $\set{(R_i,x_i,u_i)}_{i\geq 0}$ is a \emph{Schoen sequence for $\pi$ at $\mathfrak{p}$} if the following conditions are satisfied.
		
		\begin{enumerate}[label=(\roman*)]
			\item The morphism $\mc{J}_{\F_i}\colon \P^1_{\F_i} \to \P^1_{\F_i}$ is \'etale at $u_i$.
			\item The $G_{\F_i}$-action on $H^1(\P^1_{\overline{\F}}, \mc{M}_0)$ is trivial.
			\item The $(\Z/\ell)$-vector space $H^2_{u_i}(\P^1_{\overline{\F}}, \mc{M}_0)^{G_{\F_i}}$ is $1$-dimensional.
			\item The map $\delta_{u_i} \colon H^2_{u_i}(\P^1_{\overline{\F}}, \mc{M}_0)^{G_{\F_i}} \to H^1(\F_i, H^1(\P^1_{\overline{\F}}, \mc{M}_0))$ is non-zero.
			\item The restriction map $H^1(\F_j, H^1(\P^1_{\overline{\F}}, \mc{M}_0)) \to H^1(\F_j, H^1(\P^1_{\overline{\F}}, \mc{M}_0))$ is zero for all $1\leq j<i$.
			\item The point $u_i \in \mc{U}_{\F_i}$ cannot be defined over any proper subfield of $\F_i$, that is, the residue field of the morphism $u_i\colon \Spec(\F_i)\to \mc{U}$ is equal to $\F_i$.
		\end{enumerate}
	\end{defin}
	
	If $\set{(R_i,x_i,u_i)}_{i\geq 1}$ is a Schoen sequence and $z_i$ is a CM cycle above $x_i$ for all $i\geq 1$, we will show by using the maps $\Phi_i$ that the classes of the $z_i$ in $CH^2(W_{\cl{K}})/\ell$ are linearly independent. For this, we will need the following proposition, whose proof is based on the final part of the proof of \cite[Theorem 8.1]{schoen2002complex}.
	
	\begin{prop}\label{phi}
		Suppose that $\ell>5$. Let $\mathfrak{p}$ be a finite place  of $K$ of residue characteristic $p\geq 5$ such that $p\neq \ell$, and let $\set{(R_i,x_i,u_i)}_{i\geq 1}$ be a Schoen sequence for $\pi$ at $\mathfrak{p}$. For all $i\geq 1$, let $\mathfrak{p}_i$ be the maximal ideal of $R_i$. Fix a place $\cl{\mathfrak{p}}$ of $\cl{K}$ lying above all the $\mathfrak{p}_i$, and consider the maps $\Phi_i$ with respect to this choice of $\cl{\mathfrak{p}}$. For every $i\geq 1$, let $z_i$ be a CM cycle above $x_i$. 
		
		Then $\Phi_i(z_i)\neq 0$ for all $i\geq 1$ and $\Phi_i(z_j)=0$ for all $1\leq j< i$.
	\end{prop}
	
	\begin{proof}
		By \Cref{5.4}(i), the class $\theta_{x_i}([z_i])\in H^2_{x_i}(\P^1_{\cl{K}},\mc{M})^{G_{K_i}}$ is not zero. Consider the following commutative square obtained from the commutative diagram (\ref{7.9}) (for $K=K_i$, $\mathfrak{p}=\mathfrak{p}_i$, $R=R_i$, $\F=\F_i$, $D=D_i$, $t=t_i$, $x=x_i$ and $u=u_i$):
		\begin{equation}\label{7.9-smaller}
			\begin{tikzcd}
				H^2_{x_i}(\P^1_{\cl{K}},\mc{M})^{G_{K_i}} \arrow[r, "\delta_{x_i}"] \arrow[d] & H^1(K_i,H^1(\P^1_{\cl{K}},\mc{M})) \arrow[d,"r_i"] \\
				H^2_{u_i}(\P^1_{\cl{\F}},\mc{M}_0)^{G_{\F_i}} \arrow[r,"\delta_{u_i}"] & H^1(D_i,H^1(\P^1_{\cl{K}},\mc{M})).
			\end{tikzcd}
		\end{equation}
		By \Cref{schoen-sequence}(iv) the map $\delta_{u_i}$ is non-zero, and by \Cref{schoen-sequence}(iii) the $(\Z/\ell)$-vector space $H^2_{u_i}(\P^1_{\overline{\F}},\mathcal{M}_0)^{G_{\F_i}}$ is one-dimensional. We deduce that $\delta_{u_i}$ is injective and, since $\theta_{x_i}([z_i])\neq 0$, that the left vertical map in (\ref{7.9-smaller}) is bijective. The commutativity of  (\ref{7.9-smaller}) now implies that $r_i\circ\delta_{x_i}$ is injective. Thus $r_i(\delta_{x_i}(\theta_{x_i}([z_i]))\neq 0$, and hence \Cref{5.4}(ii) gives
		\[\Phi_i(z_i)=r_i(\Delta_{K_i}(\alpha^2_{\ell,K_i}(z_i)))=r_i(\delta_{x_i}(\theta_{x_i}([z_i]))\neq 0.\]
		
		Suppose now that $1\leq j< i$. By \Cref{schoen-sequence}(v), the middle vertical map in (\ref{phi-comp}) is zero. The commutativity of (\ref{phi-comp}) now implies that $\Phi_i(z_j)=0$.
	\end{proof}

	The rest of this section is devoted to the construction of Schoen sequences. The proof makes use of arguments of \cite[Proposition (8.2)]{schoen2002complex} and \cite[(8.5)]{schoen2002complex}, which are based on results of \cite{schoen1999image}. 
	
	\begin{prop}\label{schoen-exists}
		Let $\ell>5$ be a prime number, and let $\mathfrak{p}$ be a finite place  of $K$ of residue characteristic $p\geq 5$ such that $p\neq \ell$. Then there exists a Schoen sequence for $\pi$ at $\mathfrak{p}$.
	\end{prop}
	
	\begin{proof}
		Recall from \Cref{good-reduction-prop} that $\pi$ has good reduction at $\mathfrak{p}$. We first construct a tower of finite fields
		\[\F=\F_0\subset \F_1\subset \F_2\subset\cdots\]
		and $u_i\in \P^1_{\F_i}(\F_i)$ satisfying conditions (i)-(vi) of \Cref{schoen-sequence}.
		
		Let $T \subset \mc{U}_{\F}$ denote the union of the ramification locus of $\mc{J}_{\F}$ and the set of closed points $u$ such that the elliptic curve $\pi^{-1}(u)$ has supersingular $J$-invariant. By \Cref{good-reduction-prop}(2) and \cite[V.4.1(c)]{silverman2009arithmetic}, $T$ is finite. By \cite[Lemma (2.2)(ii)]{schoen2002complex}, for any finite extension $\F_i$ of $\F$, every $\F_i$-point of $(\mc{U}_{\F} - T)_{\F_i}$ is a CM point.
		
		Let $\E/\F$ be a finite field extension, let \[\kappa\colon G_{\E(\P^1)}\to \on{GL}_{\Z/\ell}(M)\] be the finite-dimensional Galois representation corresponding to the \'etale sheaf $\mc{M}_0$, let $\Gamma\subset \on{GL}_{\Z/\ell}(M)$ be the image of $\kappa$, and let $\rho\colon C\to\P^1_{\E}$ be a finite morphism from a smooth projective curve $C$ such that $\E(C)=\on{Ker}(\kappa)$. 
		
		By \cite[below (3.2.6)]{schoen2002complex}, we may choose $\F_1$ so that the hypotheses (3.2.1)-(3.2.6) of \cite[Section 3.2]{schoen1999image} are satisfied for every finite extension $\E/\F_1$. This means that for all finite extensions $\E/\F_1$ :
		\begin{enumerate}
			\item $\E$ is algebraically closed in $\E(C)$,
			\item $\E$ contains a primitive $\ell$-th root of unity,
			\item there is an $\E$-point $v$ of $\P^1_{\E}$ such that every point in $\rho^{-1}_{\E}(v)$ has degree $1$,
			\item for any two points $b_0,c_0\in \rho^{-1}_{\E}(v)$, the class of $b_0-c_0$ in $\on{Pic}(C)(\E)$ is divisible by $\ell$,
			\item the group $G_{\E}$ acts trivially on $H^1(C_{\cl{\F}},\mu_{\ell})$,
			\item if $c_0\in \rho_{\E}^{-1}(v)$ and $\breve{C}$ is defined by the Cartesian square (see \cite[(1.3.1)]{schoen1999image})
			\[
			\begin{tikzcd}
				\breve{C} \arrow[r] \arrow[d] & \on{Pic}^0(C)\arrow[d,"\times\ell"] \\
				C \arrow[r,"i_{c_0}"] & \on{Pic}^0(C),
			\end{tikzcd}
			\]
			where $i_{c_0}(c)\coloneqq \mc{O}_C(c-\deg(c)c_0)$, then the field extension $\E(\breve{C})/\E(\P^1)$ induced by the composition $\breve{C}\to C\to \P^1_{\E}$ is Galois (this is automatically true by \cite[\S 3.2, Lemma]{schoen1999image}) and for every element $\sigma$ of $\on{Gal}(\E(\breve{C})/\E(\P^1))$ there exists an $\E$-rational point $u$ of $(\mc{U}_{\F}-T)_{\E}$ such that $\rho$ is unramified at $u$ and $\sigma$ belongs to the the Frobenius conjugacy class of $x$.
		\end{enumerate}
		
		Once $\F_1$ as above is chosen, \cite[Lemma (4.4)]{schoen1999image} 
		and \cite[11.2.7]{schoen1999image} imply that for all finite extensions $\E/\F_1$ the Galois module $M$ satisfies the assumptions (3.3.1)-(3.3.6) of \cite[Section 3.3]{schoen1999image}:
		\begin{enumerate}[label=(\arabic*), start=7]
			\item The Galois module $M$ is tamely ramified,
			\item if $M^{\vee}\coloneqq \on{Hom}_{\Z}(M,\Z/\ell)$, then $(M^{\vee})^\Gamma=0$, 
			\item $M$ is an absolutely irreducible $(\Z/\ell)[\Gamma]$-module,
			\item $H^1(\Gamma,M)=0$,
			\item $H^1(\Gamma,M^{\vee})=0$, and
			\item there exists $\xi\in \Gamma$ of order prime to $\ell$ such that $M^{\ang{\xi}}$ is one-dimensional.
		\end{enumerate}
		The assumption $\ell>5$ is used here, in order to appeal to  \cite[11.2.7]{schoen1999image}. Let $\tilde{C} \to \P^1_{\F_1}$ be the Galois cover constructed in \cite[\S 7.1]{schoen1999image} (this uses property (12) above), and let $f \in \on{Gal}(\tilde{C}/\P^1_{\F_1})$ be the automorphism given by \cite[Proposition 7.2(iii)]{schoen1999image}. For every finite field extension $\F_1\subset \E\subset\cl{\F}$, we let
		\[
		\adjustbox{max width=\textwidth}{
			$N_f(\E)\coloneqq \set{u\in (\mc{U}_{\E}-T_{\E})(\E):\, \text{the Frobenius element of $u$ in  $\on{Gal}(\tilde{C}/\P^1_{\F_1})$ is conjugate to $f$}}.$ 
		}
		\]
		By \cite[Proposition 7.2(iv)]{schoen1999image}, all elements $u$ of $N_f(\E)$ satisfy properties (i)-(iv) of \Cref{schoen-sequence}:
		\begin{itemize}
			\item[(i)'] The morphism $\mc{J}_{\E}\colon \P^1_{\E} \to \P^1_{\E}$ is \'etale at $u$.
			\item[(ii)'] The $G_{\E}$-action on $H^1(\P^1_{\overline{\F}}, \mc{M}_0)$ is trivial.
			\item[(iii)'] The $(\Z/\ell)$-vector space $H^2_u(\P^1_{\overline{\F}}, \mc{M}_0)^{G_{\E}}$ is $1$-dimensional.
			\item[(iv)']The map $\delta_{u} \colon H^2_{u}(\P^1_{\overline{\F}}, \mc{M}_0)^{G_{\E}} \to H^1(\E, H^1(\P^1_{\overline{\F}}, \mc{M}_0))$ is non-zero.
		\end{itemize}
		Indeed, (i)' holds because $u$ does not belong to $T$, (ii)' holds for $\E=\F_1$ and hence for every finite extension of $\E/\F_1$, and (iii)' and (iv)' are implied by \cite[Proposition 7.2(iv)]{schoen1999image}.
		
		We write $c(f)\subset \on{Gal}(\tilde{C}/\P^1_{\F_1})$ for the  conjugacy class of $f$. If $\E/\F_1$ is a finite field extension, we write $m$ for the degree $[\E:\F_i]$. In  \cite[p. 393]{lang1956series}, Lang obtained the following estimate for $|N_f(\E)|$:
		\[|N_f(\E)| = \left(\frac{|c(f)|}{|\text{Gal}(\tilde{C}/\P^1_{\F_1})|}\right) \cdot (q^{m} + O(q^{m/2})),\qquad m\to\infty.\]
		In the formula in \cite[p. 393]{lang1956series} the fraction $\frac{h}{n}$ should be $\frac{n}{h}$.
		
		Let $N'_f(\E)\subset N_f(\E)$ be the subset of those $\E$-points $u$ satisfying (vi) of \Cref{schoen-sequence}:
		\begin{itemize}
			\item[(vi)'] $u$ is not defined over any proper subfield of $\E$.
		\end{itemize}
		As noted by Lang in the paragraph below his formula, we also have:
		\[|N'_f(\E)| = \left(\frac{|c(f)|}{|\text{Gal}(\tilde{C}/\P^1_{\F_1})|}\right) \cdot (q^{m} + O(q^{m/2})),\qquad m\to\infty.\]
		Thus, enlarging $\F_1$ if necessary, we can suppose that $N'_f(\E)$ is non-empty for every finite extension $\E/\F_1$. (Enlarging $\F_1$ is allowed because, as we explained above, any finite extension of $\F_1$ satisfies (1)-(12).) 
		
		To arrange that (v) is satisfied, it will suffice to choose the $\F_i$ such that $[\F_{i+1} : \F_i]$ is divisible by $\ell$ for each $i$. Indeed, $G_{\F_1}$ acts trivially on $N := H^1(\P^1_{\cl{\F}}, \mathcal{M}_0)$, so
		\[\text{Hom}(G_{\F_j}, N) \rightarrow \text{Hom}(G_{\F_i}, N)\]
		is the zero map when $\ell$ divides $[\F_j : \F_i]$. One may therefore construct the sequences of $\F_i$ and $u_i$ as follows: 
		\begin{itemize}
			\item[--] pick $\F_1$ such that $[\F_1:\F]$ is divisible by $\ell$ and $N'_f(\E)$ is non-empty for every finite extension $\E/\F_1$,
			\item[--] choose any sequence $\F_1\subset \F_2\subset\dots$ of finite field extensions such that $[\F_{i+1}:\F_i]$ is divisible by $\ell$ for all $i\geq i$, and
			\item[--] pick any $u_i\in N_f(\F_i)$ for all $i\geq 1$.
		\end{itemize} 
		
		We now complete the proof by showing that the $\F_i$ and $u_i$ are part of a Schoen sequence $\set{(R_i,x_i,u_i)}_{i\geq 0}$. That is, we construct an etale tower of discrete valuation rings
		\[R=R_0 \subset R_1 \subset R_2 \subset \cdots\]
		contained in $\cl{K}$ such that
		\begin{itemize}
			\item[--] the residue field of $R_i$ is $\F_i$, and
			\item[--] $u_i \in \mc{U}_{\F_i}$ lifts to a CM point $x_i \in U_{K_i}$, where $K_i$ denotes the fraction field of $R_i$.
		\end{itemize}
		
		Let $i \geq 1$. We assume that $R_{i-1}$ has already been constructed, and construct $R_i$. By induction, this will complete the proof. This construction relies on a result of Deuring on liftings of CM $J$-invariants from finite fields to number fields, which we now recall. 
		
		Since $u_i$ is a CM point, $\Lambda_i \coloneqq \on{End}(\Pi_{\F_i}^{-1}(u_i))$ is an order in the imaginary quadratic number field $F_i \coloneqq \Lambda_i \otimes_{\Z} \mathbb{Q}$. We write $J(\Lambda)$ for the $J$-invariant of the lattice $\Lambda$ and $H_i\coloneqq F_i(J(\Lambda_i))$ for the corresponding class field, which by \cite[Theorem 10.3.4]{lang1987elliptic} is a finite abelian extension of $F_i$. We write $\mathcal{O}_i$ for the ring of integers in $H_i$.
		
		\begin{prop}[Deuring]\label{deuring}
			For every prime $\wp \subset \mathcal{O}_i$ be a prime dividing $p$:
			\begin{enumerate}
				\item $\wp$ is unramified over $p$, and the residue field of $\wp$ embeds in $\F_i$, and
				\item there is an elliptic curve $E_i$ defined over $H_i$ such that $\on{End}(E_i)\cong \Lambda_i$, with good reduction at $\wp$, and such that $J(E)$ specializes to $J(\Pi_{\F_i}^{-1}(u_i))$, and
				\item if $E'_i$ is another elliptic curve over $H_i$ satisfying (2), then $J(E'_i)=J(E_i)$.
			\end{enumerate}
		\end{prop} 
		
		For the proof of \Cref{deuring}, see  \cite[Proposition (8.4)]{schoen2002complex}. Letting $\wp\subset \mc{O}_i$ be a prime dividing $p$, we obtain a commutative diagram
		\[
		\begin{tikzcd}
			\Spec(\F_i) \arrow[r,hook,"j_i"] \arrow[d] & \P^1_{\F_i} \arrow[d]\arrow[r]  & \Spec(\F_i) \arrow[d] \\
			\Spec(\mc{O}_i)    \arrow[r,hook,"J_i"] & \P^1_{\mc{O}_i} \arrow[r] & \Spec(\mc{O}_i),
		\end{tikzcd}
		\]
		where each square is cartesian, the vertical maps are induced by the composition $\mc{O}_i\to \mc{O}_i/\wp \hookrightarrow \F_i$ given by \Cref{deuring}(1), the image of $j_i$ is the $J$-invariant of $\Pi^{-1}_{\F_i}(u_i)$, and $J_i(\Spec H_i)$ is the $J$-invariant of an elliptic curve $E_i/H_i$ as in \Cref{deuring}(2). The composition of the top (resp. bottom) horizontal maps is equal to the identity of $\Spec(\F_i)$ (resp. $\on{Spec}(\mc{O}_i)$). In particular, $j_i$ and $J_i$ are closed embeddings.
		
		Choose a discrete valuation ring $R_i'$ with fraction field $H_iK_{i-1} \subset \cl{K}$ corresponding to a place above $\mathfrak{p}_{i-1}$, and write $\F'_i$ for the residue field of $R'_i$. Note that $\F_i'$ embeds in $\F_i$ because $\F_{i-1}$ and $\mathcal{O}_i/\wp$ embed in $\F_i$. We thus obtain a commutative diagram
		\begin{equation}\label{r-i-prime-diag}
			\begin{tikzcd}
				\Spec(\F_i) \arrow[r,hook,"j_i"] \arrow[d] & \P^1_{\F_i} \arrow[d]\arrow[r]  & \Spec(\F_i) \arrow[d] \\
				\Spec(R_i')    \arrow[r,hook,"J'_i"] & \P^1_{R_i'} \arrow[r] & \Spec(R_i'),
			\end{tikzcd}
		\end{equation}
		where the composition of the horizontal maps are identity maps, the squares are cartesian, and the closed embedding $J_i'$ is the pullback of $J_i$ along the morphism $\Spec(R'_i)\to \Spec(\mc{O}_i)$. 
		
		Consider the following commutative diagram with cartesian squares:
		\begin{equation}\label{cartesian}
			\begin{tikzcd}
				\cl{S}_i \arrow[r] \arrow[d] &  S_i \arrow[r,hook]\arrow[d]  & \P^1_{R'_{i-1}} \arrow[d,"\mc{J}_{R'_{i-1}}"] \\
				\Spec(\F_i) \arrow[r] & \Spec(R_i') \arrow[r,hook,"J_i'"] & \P^1_{R'_{i-1}}.
			\end{tikzcd}
		\end{equation}
		By (\ref{r-i-prime-diag}), the composition of the bottom horizontal maps is $j_i$. Therefore, since (\ref{cartesian}) is cartesian, an $\F_i$-point $v$ of $\P^1_{R'_{i-1}}$ factors through $\cl{S}_i$ if and only $\mc{J}_{\F_i}(\P^{-1}_{\F_i}(v))$ is equal to the image of $j_i$, that is, to the $J$-invariant of $\Pi^{-1}_{\F_i}(u_i)$. In particular, $u_i\colon \Spec(\F_i)\to \P^1_{R_{i-1}'}$ factors through $\cl{S}_i$.
		
		By \Cref{good-reduction-prop}(1), the morphism $\mc{J}_{R'_{i-1}}$ is finite. Thus $S_i$ is finite over $R_i'$, and in particular it is semilocal. By property (i) of \Cref{schoen-sequence}, the morphism $\mc{J}_{\F_i}$ is \'etale at $u_i$. Thus, by the openness of the \'etale locus, we may find an affine open subscheme $\Spec(R_i)\subset S_i$ containing the image of $u_i\colon \Spec(\F_i)\to S_i$ as its only closed point and such that the restriction of $\mc{J}_{R'_{i-1}}$ to $\Spec(R_i)$ is \'etale. It remains to show that $R_i$ satisfies the required properties.
		
		We first show that the inclusion $R_{i-1}\subset R_i$ is \'etale. We have the inclusions $R_{i-1}\subset R'_{i-1}\subset R_i$. We have just showed that the inclusion $R'_{i-1}\subset R_i$ is \'etale. By \Cref{deuring}(1), the inclusion $\Z_{(p)}\subset \mc{O}_i$ is \'etale, hence, so is the inclusion $R_{i-1}\subset R'_{i-1}$. Therefore the inclusion $R_{i-1}\subset R_i$ is also \'etale, as desired.
		
		We now show that the residue field of $R_i$ is equal to $\F_i$. This is true because by (\ref{cartesian}) the residue field of $R_i$ is equal to the image of the morphism $u_i\colon \Spec(\F_i)\to \P^1_{R_{i-1}'}$, and by property (vi) of \Cref{schoen-sequence} the latter is equal to $\F_i$. 
		
		Finally, we let $K_i$ be the fraction field of $R_i$, and we show that $u_i$ lifts to a CM point over $K_i$. Indeed, the composition $\Spec(K_i)\to\Spec(R_i)\to \P^1_{R'_{i-1}}$ gives rise to a $K_i$-point $x_i$ of $\P^1_{K_i}$ which lifts $u_i$. The $K_i$-point $x_i$ is a CM point by \cite[Lemma (2.2)(i)]{schoen2002complex}. 
		
		Thus $R_i$ satisfies all the required properties. By induction on $i$, this completes the proof.
	\end{proof}
	
	\section{Proof of Theorem \ref{thm-fermat}}

	\begin{prop}\label{w-infinite}
		We maintain the notation introduced in \Cref{construction}. Moreover, we suppose that $k=K=\Q(\zeta)$, where $\zeta$ is a primitive third root of unity. We let $\ell>5$ be a prime number such that $\ell\equiv 2\pmod 3$. Then the group $CH^2(W_{\cl{\Q}})/\ell$ is infinite.
	\end{prop}
	
	\begin{proof}
		Let $p\geq 5$ be a rational prime different from $\ell$, let $\mathfrak{p}$ be a place of $K$ above $p$, let $R\subset K$ be the valuation ring at $\mathfrak{p}$ and let $\F\coloneqq R/\mathfrak{p}$ be the residue field. Since $p\geq 5$, we know from \Cref{good-reduction-prop} that $\pi$ has good reduction at $\mathfrak{p}$.
		
		By \Cref{schoen-exists}, there exists a Schoen sequence $\set{(R_i,x_i,u_i)}_{i\geq 1}$ for $\pi$ at $\mathfrak{p}$. For all $i\geq 0$, let $K_i$ be the fraction field of $R_i$, let $\F_i$ be the residue field of $R_i$, and let $z_i\in Z^2_{\text{hom},\ell}(W_{K_i})$ denote a CM cycle supported on $q^{-1}(x_i)$. We aim to show that the images of the $z_i$ in $CH^2(W_{\cl{K}})/\ell$ form a linearly independent subset. Suppose the contrary. Then there exists a $\Z$-linear combination 
		\[z = \sum_{i=1}^n a_i z_i \in Z^2_{\mathrm{hom},\ell}(W_{K_n})\] such the $a_i$ are integers, the class of $z$ in $CH^2_{\mathrm{hom},\ell}(W_{\cl{K}})$ is divisible by $\ell$, but the integer $a_n$ is not divisible by $\ell$. By \Cref{phi}, we have 
		\[\Phi_n(z)=\Phi_n(a_nz_n)=a_n\Phi_n(z_n)\neq 0.\]
		Since $\Phi_n$ factors through $\alpha^2_{\ell,K_n}$, we deduce that $\alpha^2_{\ell,K_n}(z)$ is not divisible by $\ell$ in $H^1(K_n, H^3(W_{\cl{K}},\Z_{\ell}(2)))$. (Recall that $H^3(W_{\cl{K}},\Z_{\ell}(2))$ is torsion free by \cite[Proposition (4.15)]{schoen2002complex}.) 
		From (\ref{abel-jacobi-diagram}), we deduce that $\alpha^2_{\ell}(z)$ is not divisible by $\ell$ in $J^2_{\ell}(W_{\cl{K}})^{G_{K_n}}$. By \Cref{6.2}, this implies that the class of $z$ in $CH^2(W_{\cl{K}})$ is not divisible by $\ell$, a contradiction. Therefore the images of the $z_i$ in $CH^2(W_{\cl{K}})/\ell$ are linearly independent modulo $\ell$, as desired.
	\end{proof}

	\begin{proof}[Proof of \Cref{thm-fermat}]
		Let $F$ be an algebraically closed field of characteristic zero. By the Rigidity Theorem of Lecomte and Suslin \cite[Th\'eor\`eme 3.11]{lecomte1986rigidite}, the pullback map $CH^2(E^3_{\cl{\Q}})/\ell\to CH^2(E^3_F)/\ell$ is an isomorphism. We may thus assume that $F=\cl{\Q}$. 
		
		Let $k=\Q(\zeta)$, where $\zeta$ is a primitive third root of unity, let $W$ be the threefold over $k$ defined in \Cref{construction}, and fix a prime number $\ell>5$. By \Cref{schoen-thm}, we may suppose that $\ell\equiv 2\pmod 3$. 
		
		There exists a dominant rational map $f\colon E^3_{\cl{\Q}} \dashrightarrow W_{\cl{\Q}}$ of $3$-power degree; see \cite[Proposition (10.2)]{schoen2002complex}. There exist a smooth projective $\cl{\Q}$-variety $T$, a morphism $g\colon T\to E^3_{\cl{\Q}}$ given by the composition of a sequence of blow-ups at smooth centers, and a morphism $h\colon T\to W_{\cl{\Q}}$ such that $f\circ g=h$. 
		
		Since $\ell\neq 3$, by the projection formula for Chow groups the homomorphism $h^*\colon CH^2(W_{\cl{\Q}})/\ell\to CH^2(T)/\ell$ is injective. Since $\ell>5$, the group $CH^2(W_{\cl{\Q}})/\ell$ is infinite by \Cref{w-infinite}, hence so is $CH^2(T)/\ell$.
		
		The morphism $g$ is a composition of blow-ups at smooth points and curves. Thus by \cite[Proposition 6.7(e)]{fulton1998intersection} the pullback $g^*\colon CH^2(E^3_{\cl{\Q}})/\ell\to CH^2(T)/\ell$ has finite cokernel. We conclude that $CH^2(E^3_{\cl{\Q}})/\ell$ is infinite, as desired. 
	\end{proof}

	\begin{proof}[Proof of \Cref{standard-cor2}]
		To prove \Cref{standard-cor2}, let $Y$ be a smooth projective $(n-1)$-fold over $\Q$ satisfying the conclusion of \Cref{standard-cor}, let $C$ be an elliptic curve over $\Q(t)$ whose $J$-invariant is transcendental over $\Q$, and set $X\coloneqq Y_{\Q(t)}\times_{\Q(t)} C$. By a theorem of Schoen \cite[Theorem (0.2)]{schoen2000certain}, the exterior product map gives an injection $CH^i(Y_{\cl{\Q}})/\ell\hookrightarrow CH^{i+1}(X_{\cl{\Q(t)}})[\ell]$ for all prime numbers $\ell$, and so $X$ satisfies the conclusion of \Cref{standard-cor2}.
	\end{proof}
	
	\begin{proof}[Proof of \Cref{cor3}]
		To prove \Cref{cor3}, we first recall that, by a theorem of Alexandrou \cite[Theorem 1.3]{alexandrou2023torsion}, for every integer $\ell\geq 2$ there exists a smooth complex projective surface $S$ such that the torsion subgroup of N\'eron-Severi group $\on{NS}(S)_{\on{tors}}$ is cyclic of order $\ell$ and, for every smooth projective complex variety $Z$, the exterior product with any generator of $\on{NS}(S)_{\on{tors}}$ gives an injective group homomorphism $\on{Griff}^i(Z)/\ell\hookrightarrow \on{Griff}^i(S\times_{\C} Z)[\ell]/\ell \on{Griff}^i(S\times_{\C} Z)[\ell^2]$. When $\ell=2$, this is due to Schreieder \cite[Theorem 1.3]{schreieder2020infinite}. 
		
		Inspection of Alexandrou's construction, specifically of the proofs of \cite[Lemma 4.2, Theorem 4.1]{alexandrou2023torsion}, shows that there exist a surface $S_0$ over $K$ and a field inclusion $K\subset\C$ such that $(S_0)_{\C}\cong S$. By the Rigidity Theorem of Lecomte and Suslin \cite{lecomte1986rigidite}, the injectivity of the exterior product with any generator of $\on{NS}(S)_{\on{tors}}$ implies the same for $S_0$: for every smooth projective variety $Z$ over $\cl{K}$ the exterior product with any generator of $\on{NS}((S_0)_{\cl{K}})_{\on{tors}}\cong \on{NS}(S)_{\on{tors}}$ gives an injective group homomorphism $\on{Griff}^i(Z)/\ell\hookrightarrow \on{Griff}^i((S_0)_{\cl{K}}\times_{\cl{K}}Z)[\ell]/\ell \on{Griff}^i((S_0)_{\cl{K}}\times_{\cl{K}}Z)[\ell^2]$. Now, if $\ell>5$ is a prime and $E/\Q$ is the Fermat cubic, \Cref{thm-fermat} implies that $\on{Griff}^2(E^3_{\cl{K}})/\ell$ is infinite, and hence  $X=(E^3)_K\times_K\P^{n-5}_K\times_KS$ satisfies the conclusion of \Cref{cor3}.
	\end{proof}

	\section{Acknowledgements}
	
	I thank Akhil Mathew and Madhav Nori for an inspiring conversation, held in Chicago in April 2023, on the topic of \cite{farb2021prismatic} and Brosnan's Conjecture, and Burt Totaro for his interest in this work and for helpful comments. I thank Stefan Schreieder for his suggestion to include \Cref{cor3} in the paper.
	

\end{document}